\documentclass[reqno]{amsart}

\usepackage{amssymb}
\usepackage{hyperref}
\usepackage{mathtools}
\usepackage[all]{xy}
\usepackage{ifthen}
\usepackage{xargs}
\usepackage{bussproofs}

\hypersetup{colorlinks=true,linkcolor=blue}

\newcommand{\newref}[4][]{
\ifthenelse{\equal{#1}{}}{\newtheorem{h#2}[hthm]{#4}}{\newtheorem{h#2}{#4}[#1]}
\expandafter\newcommand\csname r#2\endcsname[1]{#3~\ref{#2:##1}}
\expandafter\newcommand\csname R#2\endcsname[1]{#4~\ref{#2:##1}}
\expandafter\newcommand\csname n#2\endcsname[1]{\ref{#2:##1}}
\newenvironmentx{#2}[2][1=,2=]{
\ifthenelse{\equal{##2}{}}{\begin{h#2}}{\begin{h#2}[##2]}
\ifthenelse{\equal{##1}{}}{}{\label{#2:##1}}
}{\end{h#2}}
}

\newref[section]{thm}{theorem}{Theorem}
\newref{lem}{lemma}{Lemma}
\newref{prop}{proposition}{Proposition}
\newref{cor}{corollary}{Corollary}
\newref{cond}{condition}{Condition}
\newref{conj}{conjecture}{Conjecture}

\theoremstyle{definition}
\newref{defn}{definition}{Definition}
\newref{example}{example}{Example}

\theoremstyle{remark}
\newref{rem}{remark}{Remark}

\newcommand{\cat}[1]{\mathbf{#1}}
\newcommand{\C}{\cat{C}}

\newcommand{\Set}{\cat{Set}}
\newcommand{\sSet}{\cat{sSet}}
\newcommand{\qcat}{\cat{qcat}}
\newcommand{\K}{$\mathcal{K}$}
\newcommand{\join}{\star}
\newcommand{\fjoin}{\diamond}
\newcommand{\Hom}{\mathrm{Hom}}
\newcommand{\Map}{\mathrm{Map}}
\newcommand{\Fun}{\mathrm{Fun}}
\newcommand{\colim}{\mathrm{colim}}

\newcommand{\we}{\mathcal{W}}
\newcommand{\I}{\mathrm{I}}
\newcommand{\J}{\mathrm{J}}
\newcommand{\class}[2]{#1\text{-}\mathrm{#2}}
\newcommand{\Icell}[1][\I]{\class{#1}{cell}}
\newcommand{\Icof}[1][\I]{\class{#1}{cof}}
\newcommand{\Iinj}[1][\I]{\class{#1}{inj}}
\newcommand{\Jinj}[1][]{\Iinj[\J#1]}
\newcommand{\Jcell}[1][]{\Icell[\J#1]}
\newcommand{\Jcof}[1][]{\Icof[\J#1]}

\numberwithin{figure}{section}

\newcommand{\pb}[1][dr]{\save*!/#1-1.2pc/#1:(-1,1)@^{|-}\restore}
\newcommand{\po}[1][dr]{\save*!/#1+1.2pc/#1:(1,-1)@^{|-}\restore}

\begin{document}

\title{Model category of marked objects}

\author{Valery Isaev}

\begin{abstract}
For every functor $\mathcal{F} : \mathcal{K} \to \mathbf{C}$, where $\mathcal{K}$ is a small category and $\mathbf{C}$ is a model category which satisfies some mild hypotheses,
we define a model category $\mathbf{C}^m$ of $\mathcal{K}$-marked objects of $\mathbf{C}$.
We consider an application of this construction to the category of simplicial sets with the Joyal model structure.
Marked simplicial sets can be thought of as $(\infty,1)$-categories with some additional structure which depends on $\mathcal{F}$.
In particular, we construct a model category of quasi-categories which have limits of all diagrams of any given shape.
\end{abstract}

\maketitle

\section{Introduction}

Model categories, introduced in \cite{quillen}, are an important tool in homotopy theory and higher category theory.
For every model category $\C$ which satisfies some mild hypotheses and every functor $\mathcal{F} : \mathcal{K} \to \C$, we define a model category $\C^m$ of marked objects.
A marked object is an object of $\C$ in which some maps of the form $\mathcal{F}(K) \to X$ are marked (see \rdefn{marked-obj} for a precise definition).

If $\C$ is a category of simplicial sets with the Joyal model structure (see \cite{Joyal,lurie-topos})
and $\mathcal{K}$ is a set of objects of the form $\Delta^0 \join L$ for some simplicial set $L$,
then a marked simplicial set is a simplicial set in which some cones of some diagrams are marked.
The idea is that a cone should be marked if and only if it is a limit cone.
Of course, in general, this is might not hold, but we will construct a model structure on the category of marked simplicial sets in which fibrant objects have this property (see \rprop{mark-fib-obj}).
Moreover, every diagram of a given shape in a fibrant marked simplicial set has a limit.
Thus, this model category represents the $(\infty,1)$-category of $(\infty,1)$-categories in which limits of certain diagrams exist (and functors between $(\infty,1)$-categories that preserve them).

Of course, there is a dual model category of simplicial sets with marked cocones which represents the $(\infty,1)$-category of $(\infty,1)$-categories in which colimits of certain diagrams exist.
By choosing another category $\mathcal{K}$, it should be possible to construct model categories that represent $(\infty,1)$-categories of $(\infty,1)$-categories with other categorical structures,
such as the $(\infty,1)$-category of (locally) Cartesian closed $(\infty,1)$-categories.
We will not discuss such constructions in this paper.

In section~\ref{sec:marked}, we define categories of marked objects and prove some of their properties.
In section~\ref{sec:model-structure}, we construct model structures on categories of marked objects.
In section~\ref{sec:marked-simp-sets}, we consider a localization of a model category of marked simplicial sets and give a characterization of fibrant objects in this model category.

\section{Category of marked objects}
\label{sec:marked}

In this section for every combinatorial model category $\C$, we define a new model category $\C^m$ of marked objects of $\C$.
This model category is usually not useful by itself.
The idea is that we should take a left Bousfield localization of $\C^m$ to obtain an interesting model category.
We will show examples of this construction in the next section.

\begin{defn}[marked-obj]
Let $\C$ be a category, let $\mathcal{K}$ be a small category, and let $\mathcal{F} : \mathcal{K} \to \C$ be a functor.
A \emph{\K-marked object} of $\C$ is a pair $(X,\mathcal{E})$ where $X$ is an object of $\C$ and $\mathcal{E} : \mathcal{K}^{op} \to \Set$ is a subfunctor of $\Hom(\mathcal{F}(-),X)$.
Morphisms $f : \mathcal{F}(K) \to X$ that belong to $\mathcal{E}$ will be called \emph{marked}.
A morphism of marked objects is a morphism of the underlying objects that preserves marked morphisms.
The category of marked objects will be denoted by $\C^m$.
\end{defn}

We will sometimes omit mention of $\mathcal{F}$ and identify an object $K$ of $\mathcal{K}$ with its image in $\C$.

Let $S$ be a set of maps of the form $\mathcal{F}(K) \to X$.
Then we will write $GS : \mathcal{K}^{op} \to \Set$ for the subfunctor of $\Hom(\mathcal{F}(-),X)$ generated by $S$.
A map $f : \mathcal{F}(K) \to X$ belongs to $GS$ if and only if it factors as $\mathcal{F}(K) \xrightarrow{\mathcal{F}(k)} \mathcal{F}(K') \xrightarrow{f'} X$
for some $k : K \to K'$ and some $f' \in S$.

Forgetful functor $U : \C^m \to \C$ has a left adjoint $(-)^\flat : \C \to \C^m$ and a right adjoint $(-)^\sharp : \C \to \C^m$.
For every $X \in \C$, $X^\flat$ is the marked object in which no morphisms are marked (that is, $X^\flat = (X,\varnothing)$),
and $X^\sharp$ is the marked object in which all morphisms are marked (that is, $X^\sharp = (X,\coprod_{K \in \mathcal{K}} \Hom(K,X))$).
Objects of the form $X^\flat$ and of the form $X^\sharp$ will be called \emph{flat} and \emph{sharp} respectively.

Category $\C^m$ has the same limits and colimits as $\C$.
Let $D : J \to \C^m$ be a diagram.
Then underlying objects of $\lim(D)$ and $\colim(D)$ are a limit and a colimit of underlying objects in $\C$ respectively.
Morphism $K \to \lim(D)$ is marked if and only if morphism $K \to \lim(D) \to D_j$ is marked for every $j \in J$.
Morphism $K \to \colim(D)$ is marked if and only if it factors through some marked morphism $K \to D_j$.

\begin{prop}[mark-comb]
If $\C$ is locally presentable, then so is $\C^m$.
\end{prop}
\begin{proof}
First, let us prove that there exists a set of objects $S^m$ of $\C^m$ which generates the whole category under colimits.
Let $\lambda$ be a regular cardinal such that $\C$ is locally $\lambda$-presentable and for every $K \in \mathcal{K}$, $\mathcal{F}(K)$ is $\lambda$-presentable.
Let $S$ be a set of objects of $\C$ such that every object of $X$ is a $\lambda$-filtered colimit of objects from $S$.
Note that for every object $X$ of $\C$, there is only a set of marked objects $Y$ such that $U(Y) = X$.
Let $S^m$ be the set of objects $Y$ such that $U(Y) \in S$.
Let $X$ be a marked object, and let $D : J \to \C$ be a $\lambda$-filtered diagram such that $\colim(D) = U(X)$.
Then define diagram $D' : J \to \C^m$ so that $U(D'_j) = D_j$ and $f : K \to D_j$ is marked if and only if $K \xrightarrow{f} D_j \to \colim(D) = U(X)$ is marked.
To prove that $\colim(D') = X$, we need to show that $f : K \to U(X)$ is marked in $X$ only if it factors through some $D_j \to U(X)$.
But this follows from the fact that $D$ is $\lambda$-filtered and $K$ is $\lambda$-presentable.

Now, let us prove that every object of $\C^m$ is small.
Let $(X,\mathcal{E})$ be a marked object, and let $\lambda$ be a regular cardinal such that
for every $K \in \mathcal{K}$, $\mathcal{F}(K)$ is $\lambda$-presentable, $X$ is $\lambda$-presentable, and $|\mathcal{E}| < \lambda$.
Let $D : J \to \C^m$ be a $\lambda$-filtered diagram.
It is easy to see that $\colim_{j \in J} \Hom((X,\mathcal{E}), D_j) \to \Hom((X,\mathcal{E}), \colim_{j \in J} D_j)$ is injective.
Let us show that it is surjective.
Let $f : (X,\mathcal{E}) \to \colim_{j \in J} D_j$ be a map of marked objects.
Then $U(f)$ factors as $X \xrightarrow{g} U(D_j) \to \colim_{j \in J} U(D_j)$ for some $j$.
By the description of colimits that we gave above, for every marked $k : K \to X$, map $U(f) \circ k$ factors through some marked $h_k : K \to U(D_k)$.
Since $K$ is $\lambda$-presentable and maps $K \xrightarrow{h_k} U(D_k) \to \colim_{j \in J} U(D_j)$
and $K \xrightarrow{k} X \xrightarrow{g} U(D_j) \to \colim_{j \in J} U(D_j)$ are equal,
there exists an object $D'_k$ and maps $D_j \to D'_k$ and $D_k \to D'_k$ in the diagram such that
$K \xrightarrow{h_k} U(D_k) \to U(D'_k)$ and $K \xrightarrow{k} X \xrightarrow{g} U(D_j) \to U(D'_k)$ are also equal.
In particular, $K \xrightarrow{k} X \xrightarrow{g} U(D_j) \to U(D'_k)$ is marked.
Finally, since $D$ is $\lambda$-filtered and $|\mathcal{E}| < \lambda$, there exists an object $D_i$ together with maps $D'_k \to D_i$ for every $k$.
For every $k : K \to X$, map $K \xrightarrow{k} X \xrightarrow{g} U(D_j) \to U(D_i)$ is marked.
Hence $X \xrightarrow{g} U(D_j) \to U(D_i)$ is a morphism of marked objects.
Thus $f : (X,\mathcal{E}) \to \colim_{j \in J} D_j$ factors through $D_i \to \colim_{j \in J} D_j$.
\end{proof}

Now, we assume that there is a structure of a model category on $\C$ such that every object in the image of $\mathcal{F}$ is cofibrant.
We will say that a map $f : X \to Y$ of marked objects is a \emph{cofibration} if and only if $U(f)$ is a cofibration in $\C$. 
\begin{prop}
If $\I$ is a set of generating cofibrations of $\C$, then cofibrations of $\C^m$ are generated by the set $\I^m$ which consists of the following maps:
\begin{enumerate}
\item Map $i^\flat$ for every $i \in \I$.
\item Map $K^\flat \to (K, G \{ id : K \to K \})$ for every $K \in \mathcal{K}$.
\end{enumerate}
\end{prop}
\begin{proof}
Since every map in $\I^m$ is a cofibration and $U$ preserves colimits, $\Icof[\I^m]$ consists of cofibrations.
Let us show that every cofibration $f : X \to Y$ belongs to $\Icof[\I^m]$.
First, assume that $U(f)$ is an isomorphism.
Then $f$ is the following pushout:
\[ \xymatrix{ \coprod\limits_\mathcal{E} K^\flat \ar[r] \ar[d] & X \ar[d] \\
              \coprod\limits_\mathcal{E} (K, G \{id\}) \ar[r] & \po Y,
            } \]
where $\mathcal{E}$ is the set of marked maps of $Y$.

Now, assume that $U(f)$ is a relative $\Icell$ complex.
Then it is a transfinite composition $X_0 \to X_\lambda$, where for each $\alpha$, $X_\alpha \to X_{\alpha+1}$ is a pushout of a map $U_\alpha \to V_\alpha$ from $\I$.
We define a transfinite sequence $X = X'_0 \to \ldots \to X'_\lambda$ such that $U(X'_\alpha \to X'_\beta)$ equals to $X_\alpha \to X_\beta$ for every $\alpha$ and $\beta$.
Map $X'_\alpha \to X'_{\alpha+1}$ is defined as a pushout of $U^\flat_\alpha \to V^\flat_\alpha$.
Then $f$ factors as $X \to X'_\lambda \xrightarrow{g} Y$ so that $U(g)$ is an isomorphism.
Hence $f$ is a relative $\Icell[\I^m]$ complex.

Finally, in the general case, $U(f)$ is a retract of a relative $\Icell$ complex $U(X) \to Z$.
Let $Z'$ be the marked object such that $U(Z') = Z$ and $K \to Z$ is marked if and only if it factors through $Y \to Z$.
Then $f$ is a retract of $X \to Z'$.
Hence $f$ belongs to $\Icof[\I^m]$.
\end{proof}

Let $\mathcal{J}$ be a set of cofibrant marked objects.
We are going to define a model structure on $\C^m$ such that $\I^m$ is a set of generating cofibrations,
and for every $J \in \mathcal{J}$, map $U(J)^\flat \to J$ is a trivial cofibration.

Let us define a fibrant replacement functor $R^m$ for $\C^m$.
Let $R : \C \to \C$, $t_X : X \to R(X)$ be a fibrant replacement functor for $\C$.
For every marked $X$, underlying object of $R^m(X)$ is $R(U(X))$, and a map $K \to R(U(X))$ is marked if and only if it is homotopic to a map that
either factors as $K \to U(X) \xrightarrow{t_{U(X)}} R(U(X))$, where the first map is marked,
or factors through a marked map $K \to U(J)$ for some $J \in \mathcal{J}$.

We will say that a map $f$ of marked objects is a \emph{weak equivalence}
if and only if the underlying map of $f$ is a weak equivalence and $R^m(f)$ reflects marked maps.

\begin{prop}[2-out-of-3]
Weak equivalences of marked objects satisfy 2-out-of-3 property.
\end{prop}
\begin{proof}
Let $f : X \to Y$ and $g : Y \to Z$ be maps of marked objects such that $U(f)$ and $U(g)$ are equivalences.
If $R^m(f)$ and $R^m(g)$ reflect marked maps, then so is $R^m(g \circ f)$.
If $R^m(g \circ f)$ reflects marked maps, then so is $R^m(f)$.
Let us prove that if $R^m(g \circ f)$ reflects marked maps, then so is $R^m(g)$.
Let $k : K \to R(U(Y))$ be a map such that $R(U(g)) \circ k$ is marked.
Since $K$ is cofibrant and $R(U(f))$ is a weak equivalence between fibrant objects,
there exists a map $k' : K \to R(U(X))$ such that $R(U(f)) \circ k'$ is homotopic to $k$.
Since $R(U(g)) \circ k$ and $R(U(g \circ f)) \circ k'$ are homotopic and $R(U(g)) \circ k$ is marked, $R(U(g \circ f)) \circ k'$ is also marked.
Since $g \circ f$ reflects marked maps, $k'$ is marked.
Since $R(U(f)) \circ k'$ and $k$ are homotopic and $R(U(f)) \circ k'$ is marked, $k$ is also marked.
\end{proof}

The following lemma gives a useful explicit description of weak equivalences.

\begin{lem}[we]
If $f : X \to Y$ is a map of marked objects such that $U(f)$ is a weak equivalence,
then it is a weak equivalence if and only if for every $h : K \to R(U(X))$ and every marked $k : K \to U(Y)$,
if $R(U(f)) \circ h$ and $t_{U(Y)} \circ k$ are homotopic, then $h$ is marked in $R^m(X)$.
\end{lem}
\begin{proof}
By the definition, $f$ is a weak equivalence if and only if for every $h : K \to R(U(X))$ such that $R(U(f)) \circ h$ is marked in $Y$, $h$ is marked in $R^m(X)$.
Map $R(U(f)) \circ h$ is marked in $R^m(Y)$ if and only if either there exists marked map $k : K \to U(Y)$ such that $t_{U(Y)} \circ k$ is homotopic to $R(U(f)) \circ h$,
or there exist $J \in \mathcal{J}$, a marked map $k : K \to U(J)$ and a map $g : U(J) \to R(U(Y))$ such that $g \circ k$ is homotopic to $R(U(f)) \circ h$.
We just need to prove that if the second case holds, then $h$ is marked in $R^m(X)$.

Since $U(J)$ is cofibrant and $R(U(f))$ is a weak equivalence between fibrant objects,
there exists a map $g' : U(J) \to R(U(X))$ such that $R(U(f)) \circ g'$ is homotopic to $g$.
Hence $R(U(f)) \circ h$ is homotopic to $R(U(f)) \circ g' \circ k$.
Since $R(U(f))$ is a weak equivalence between fibrant objects, this implies that $h$ is homotopic to $g' \circ k$.
\end{proof}

\begin{prop}[iinj]
Every map $f : X \to Y \in \Iinj[\I^m]$ is a weak equivalence.
\end{prop}
\begin{proof}
Let $k : K \to U(Y)$ be a marked map, and let $h : K \to R(U(X))$ be a map such that $R(U(f)) \circ h$ and $t_{U(Y)} \circ k$ are homotopic.
Since $(K, G \{id\})$ is cofibrant, we have a lift in the following diagram:
\[ \xymatrix{                                    & X \ar[d]^f \\
              (K, G \{id\}) \ar[r] \ar@{-->}[ur] & Y
            } \]
Hence there exists a marked map $k' : K \to U(X)$ such that $U(f) \circ k' = k$.
Note that $R(U(f)) \circ t_{U(X)} \circ k' = t_{U(Y)} \circ U(f) \circ k' = t_{U(Y)} \circ k \sim R(U(f)) \circ h$.
Since $R(U(f))$ is a weak equivalence between fibrant objects, this implies that $t_{U(X)} \circ k'$ is homotopic to $h$.
Hence $h$ is marked in $R^m(X)$, and by the previous lemma, $f$ is a weak equivalence.
\end{proof}

The following lemma gives a useful characterization of trivial cofibrations.

\begin{lem}
Let $f : X \to Y$ be a cofibration of marked objects such that $U(f)$ is a weak equivalence.
Let $g : U(Y) \to R(U(X))$ be a map such that $g \circ U(f) = t_{U(X)}$.
Then $f$ is a weak equivalence if and only if for every marked $k : K \to U(Y)$, map $g \circ k$ is marked in $R^m(X)$.
\end{lem}
\begin{proof}
First, suppose that $f$ is a weak equivalence.
Let $k : K \to U(Y)$ be a marked map.
Since $U(f)$ is a trivial cofibration, $R(U(f)) \circ g \circ k$ and $t_{U(Y)} \circ k$ are homotopic.
Hence, by \rlem{we}, map $g \circ k$ is marked.
Conversely, let $k : K \to U(Y)$ be a marked map, and let $h : K \to R(U(X))$ be a map such that $R(U(f)) \circ h$ is homotopic to $t_{U(Y)} \circ k$.
Note that $R(U(f)) \circ h \sim t_{U(Y)} \circ k \sim R(U(f)) \circ g \circ k$.
Since $R(U(f))$ is a weak equivalence between fibrant objects, this implies that $h$ is homotopic to $g \circ k$.
Since $g \circ k$ is marked, $h$ is also marked.
\end{proof}

Now, we need to establish the following simple lemma:

\begin{lem}
Let $f : X \to Y$ be a map of marked objects, and let $g : R(U(X)) \to R(U(Y))$ be a map such that the following diagram commutes:
\[ \xymatrix{ U(X) \ar[r]^{U(f)} \ar[d]_{t_{U(X)}} & U(Y) \ar[d]^{t_{U(Y)}} \\
              R(U(X)) \ar[r]_g & R(U(Y))
            } \]
Then $g$ preserves marked maps.
\end{lem}
\begin{proof}
Let $k : K \to R(U(X))$ be a marked map.
Then either there exists a marked map $k' : K \to U(X)$ such that $t_{U(X)} \circ k'$ is homotopic to $k$,
or there exist $J \in \mathcal{J}$, a marked map $k' : K \to U(J)$ and a map $h : U(J) \to R(U(X))$ such that $h \circ k'$ is homotopic to $k$.
If the first case holds, then $U(f) \circ k'$ is marked and $t_{U(Y)} \circ U(f) \circ k'$ is homotopic to $g \circ k$.
If the second case holds, then $g \circ h \circ k'$ is homotopic to $g \circ k$.
Thus $g \circ k$ is marked.
\end{proof}

\begin{prop}[pushouts-trans-comp]
Trivial cofibrations are closed under pushouts and transfinite compositions.
\end{prop}
\begin{proof}
First, let us prove the latter.
Let $c : X_0 \to X_\lambda$ be a transfinite composition of a sequence $X : \lambda \to \C^m$.
By \rprop{2-out-of-3}, weak equivalences are closed under composition.
Thus we may assume that $\lambda$ is a limit ordinal.
Let $k : K \to U(X_\lambda)$ be a marked map, and let $g : U(X_\lambda) \to R(U(X_0))$ be a map such that $g \circ c = t_{U(X_0)}$.
Since $k$ is marked, it factors through some marked map $k' : K \to U(X_\alpha)$.
Since $X_0 \to X_\alpha$ is a trivial cofibration, map $K \xrightarrow{k'} U(X_\alpha) \to U(X_\lambda) \xrightarrow{g} R(U(X_0))$ is marked.
Hence $c$ is a trivial cofibration.

Now, let us prove that trivial cofibrations are closed under pushouts.
Consider the following pushout square, where $f$ is a trivial cofibration:
\[ \xymatrix{ X \ar[r]^p \ar[d]_f & Z \ar[d]^{f'} \\
              Y \ar[r]_q & \po T
            } \]
Let $k' : K \to U(T)$ be a marked map, and let $g' : U(T) \to R(U(Z))$ be a map such that $g' \circ U(f') = t_{U(Z)}$.
Then $k'$ factors through either $U(Z)$ or $U(Y)$.
If $k'$ factors through $U(Z)$, then this immediately implies that $g' \circ k'$ is marked.
If $k'$ factors through a marked map $k : K \to U(Y)$, then there exists a map $g : U(Y) \to R(U(X))$ such that $g \circ U(f) = t_{U(X)}$ and $g \circ k$ is marked.
By the previous lemma, $R(U(p)) \circ g \circ k$ is marked.
Since $f$ is a trivial cofibration, $R(U(p)) \circ g$ is homotopic to $g' \circ U(q)$.
Hence $g' \circ U(q) \circ k = g' \circ k'$ is also marked.
\end{proof}

\section{Model structure on the category of marked objects}
\label{sec:model-structure}

To construct a model structure on $\C^m$, we will need the following theorem by Jeff Smith (see, for example, \cite[Proposition~A.2.6.8]{lurie-topos}):
\begin{thm}[mod-comb]
Let $\C$ be a locally presentable category, let $\I$ be a set of maps of $\C$, and let $\we$ be a class of maps of $\C$.
Suppose that the following conditions hold:
\begin{enumerate}
\item The intersection $\Icof \cap \we$ is closed under pushouts and transfinite compositions.
\item The full subcategory $\we$ of the category of arrows of $\C$ is accessible subcategory.
\item The class $\we$ has the 2-out-of-3 property.
\item $\Iinj \subseteq \we$.
\end{enumerate}
Then there exists a cofibrantly generated model structure on $\C$ with $\Icof$ as the class of cofibrations and $\we$ as the class of weak equivalences.
\end{thm}

We will also need the following theorem (see, for example, \cite[Theorem~2.1.19]{hovey}):
\begin{thm}[mod-cof]
Suppose that $\C$ is a complete and cocomplete category, $\we$ is a class of morphisms of $\C$, and $\I$, $\J$ are sets of morphisms of $\C$.
Then $\C$ is a cofibrantly generated model category with $\I$ as the set of generating cofibrations,
$\J$ as the set of generating trivial cofibrations, and $\we$ as the class of weak equivalences if and only if the following conditions are satisfied:
\begin{enumerate}
\item The domains of $\I$ and $\J$ are small relative to $\Icell$.
\item $\we$ has 2-out-of-3 property and is closed under retracts.
\item $\Iinj \subseteq \we$.
\item $\Jcell \subseteq \we \cap \Icof$.
\item Either $\Jinj \cap \we \subseteq \Iinj$ or $\Icof \cap \we \subseteq \Jcof$.
\end{enumerate}
\end{thm}

Now, we can prove our main result:

\begin{thm}[mark-main]
Let $\C$ be a model category, let $\mathcal{K}$ be a multiset of cofibrant objects of $\C$, and let $\mathcal{J}$ be a set of cofibrant marked objects.
Suppose that one of the following conditions hold:
\begin{enumerate}
\item \label{it:comb} $\C$ is a combinatorial model category.
\item \label{it:all-fib} Every object of $\C$ is fibrant and there exists a set $\I$ of generating cofibrations such that
the domains and the codomains of maps in $\I$ are small relative to $\Icell$.
\end{enumerate}
Then there exists a cofibrantly generated model structure on $\C^m$.
Both adjoint pairs $(-)^\flat \dashv U$ and $U \dashv (-)^\sharp$ are Quillen pairs.
If $\C$ is left proper, then so is $\C^m$.
A marked object $X$ is fibrant in $\C^m$ if and only if the following conditions hold:
\begin{itemize}
\item Underlying object $U(X)$ is fibrant in $\C$
\item For every $J \in \mathcal{J}$, $X$ has RLP with respect to $U(J)^\flat \to J$.
\item Marked maps in $X$ are stable under homotopy (that is, if two maps $K \to U(X)$ are homotopic and one of them is marked, then so is the other).
\end{itemize}
\end{thm}
\begin{proof}
Every map in $\Iinj[\I^m]$ is a weak equivalence by \rprop{iinj},
trivial cofibrations are stable under pushouts and transfinite compositions by \rprop{pushouts-trans-comp},
weak equivalences has 2-out-of-3 property by \rprop{2-out-of-3},
and it is easy to see that they are closed under retracts.

First, assume that \eqref{it:comb} holds.
By \rprop{mark-comb}, $\C^m$ is locally presentable.
By \rthm{mod-comb}, we just need to prove that the class of weak equivalences is an accessible subcategory of $\C^m$.
Let $\J$ be a set of generating trivial cofibrations in $\C$.
Then we can take $R$ to be the functor obtained form the small object argument for $\J$.
Let $\J_\mathcal{K}$ be the set of maps $(C(K), G \{i_0\}) \to (C(K), G \{i_0,i_1\})$ for every $K \in \mathcal{K}$, where $[i_0,i_1] : K \amalg K \to C(K)$ is a cylinder object for $K$.
Let $\J_0$ be the set of maps $U(J)^\flat \to J$ for every $J \in \mathcal{J}$.
Then $R^m$ can be described as the composition of two functors $R_1$ and $R_2$.
The former is obtained from the small object argument for $\J^\flat$ and the latter from the small object argument for $\J_\mathcal{K} \cup \J_0$.
Thus $R^m$ preserves $\kappa$-filtered colimits for some regular cardinal $\kappa$.

Let $\C^m_0$ be the full subcategory of the arrow category of $\C^m$ on maps $f$ such that $U(f)$ is a weak equivalence and $f$ reflects marked maps.
Since $R^m$ preserves $\kappa$-filtered colimits, by \cite[Corollary~A.2.6.5]{lurie-topos},
to prove that the class of weak equivalences is an accessible subcategory of $\C^m$,
we just need to show that $\C^m_0$ is closed under $\kappa$-filtered colimits.
But this is obvious if we take $\kappa$ such that the class of weak equivalences in $\C$
is closed under $\kappa$-filtered colimits and every object in $\mathcal{K}$ is $\kappa$-presentable.

Now, assume that \eqref{it:all-fib} holds.
By \cite[Corollary~3.2]{f-model-structures}, there exists a set $\J_\I$ of generating trivial cofibrations for $\C$.
For every $J \in \mathcal{J}$ and marked $k : K \to U(J)$, let $Z_{J,k}$ be the following pushout:
\[ \xymatrix{ K \ar[r]^k \ar[d]_{i_0} & U(J) \ar[d] \\
              C(K) \ar[r]_g & \po Z_{J,k}
            } \]
Let $\J_1$ be the set of maps of the form $Z_{J,k}^\flat \to (Z_{J,k}, G\{ g \circ i_1 \})$ for every $J \in \mathcal{J}$ and marked $k : K \to U(J)$.
Then $\J^m = \J_\I^\flat \cup \J_\mathcal{K} \cup \J_0 \cup \J_1$ is a set of generating trivial cofibrations for $\C^m$.
Indeed, every map in $\J^m$ is cofibration and a weak equivalence.
Every map $f \in \J_\I^\flat$ is a weak equivalence since its underlying map is a weak equivalence and its codomain is flat.
Every map $f \in \J_\mathcal{K} \cup \J_1$ is a weak equivalence since $R^m(f)$ is an identity.

By \rthm{mod-cof}, we just need to prove that every weak equivalence that has RLP with respect to $\J^m$ also has RLP with respect to $\I^m$.
Let $f : X \to Y$ be such a map.
Since $U(f)$ is a weak equivalence and a fibration, it has RLP with respect to $\I$.
Hence $f$ has RLP with respect to $\I^\flat$.
Since every object of $\C$ is fibrant, we can take $R$ to be the identity functor.
Let $k : K \to U(X)$ be a map such that $U(f) \circ k$ is marked.
Since $f$ is a weak equivalence, $k$ is marked in $R^m(X)$.
Thus either there exists a marked map $k' : K \to U(X)$ which is homotopic to $k$,
or there exist $J \in \mathcal{J}$, a marked map $k' : K \to U(J)$ and $h : U(J) \to U(X)$ such that $h \circ k'$ is homotopic to $k$.
If the first case holds, then $U(f) \circ k$ and $U(f) \circ k'$ are homotopic and are both marked.
Since $f$ has RLP with respect to $\J_\mathcal{K}$ and $k'$ is marked, $k$ is also marked.
If the second case holds, then there is a map $z : Z_{J,k'} \to U(X)$ such that $z \circ g \circ i_0 = k$.
Since $f$ has RLP with respect to $\J_1$ and $U(f) \circ k$ is marked, $k$ is also marked.
Thus $f$ has RLP with respect to $K^\flat \to (K, G \{id\})$ for every $K \in \mathcal{K}$.

This completes the construction of the model structure.
Now, let us consider the characterization of fibrant objects in this model category.
We need to prove that a marked object is fibrant if and only if it has RLP with respect to $\J^\flat \cup \J_\mathcal{K} \cup \J_0$,
where $\J$ is a set of generating trivial cofibrations for $\C$.
The ``only if'' direction follows from the fact that $\J^\flat \cup \J_\mathcal{K} \cup \J_0$ consists of trivial cofibrations.
Conversely, let $Z$ be a marked object that has RLP with respect to $\J^\flat \cup \J_\mathcal{K} \cup \J_0$.
Let $f : X \to Y$ be a trivial cofibration, and let $h : X \to Z$ be a map.
Since $U(Z)$ is fibrant, $U(h)$ factors through $t_{U(X)}$:
\[ \xymatrix{ U(X) \ar[r]^-{t_{U(X)}} \ar[d]_{U(f)} & R(U(X)) \ar[r]^h & U(Z) \\
              U(Y) \ar@{-->}[ur]_g
            } \]

Since $U(f)$ is a trivial cofibration, there exists a lift $g : U(Y) \to R(U(X))$ in the diagram above.
We just need to prove that $h \circ g$ preserves marked maps.
Let $k : K \to U(Y)$ be a marked map.
Since $f$ is a trivial cofibration, either there exists a marked map $k' : K \to U(X)$ such that $t_{U(X)} \circ k'$ is homotopic to $g \circ k$,
or there exist $J \in \mathcal{J}$, a marked map $k' : K \to U(J)$, and a map $j : U(J) \to R(U(X))$ such that $j \circ k'$ is homotopic to $g \circ k$.
If the first case holds, then $h \circ g \circ k$ is marked since marked maps in $Z$ are stable under homotopy.
If the second case holds, then $h \circ j \circ k'$ is marked since $Z$ has RLP with respect to $U(J)^\flat \to J$.
Hence $h \circ g \circ k$ is also marked.
Thus $h \circ g$ preserves marked maps.

The fact that both adjoint pairs of functors $(-)^\flat \dashv U$ and $U \dashv (-)^\sharp$ are Quillen adjunctions
easily follows from the fact that both $(-)^\flat$ and $U$ preserve cofibrations and weak equivalences.

Finally, let us prove that if $\C$ is left proper, then so is $\C^m$.
Let $Z \to T$ be a pushout of a weak equivalence $f : X \to Y$ along a cofibration.
We may assume that $f$ is a trivial fibration.
Let $h : K \to R(U(Z))$ be a map, and let $k : K \to U(T)$ be a marked map.
Then $k$ factors through either $U(Z)$ or $U(Y)$.
In the first case, $K \to U(Z) \xrightarrow{t_{U(Z)}} R(U(Z))$ is homotopic to $h$.
Since the first map is marked, $h$ is also marked.
In the second case, we have a marked map $k' : K \to U(Y)$.
Since $f : X \to Y$ is a trivial fibration, $k'$ lifts to a marked map $K \to U(X)$.
Then $K \to U(X) \to U(Z) \to R(U(Z))$ is homotopic to $h$.
Since the first map is marked, $h$ is also marked.
Thus, in either case, $h$ is marked; hence $Z \to T$ is a weak equivalence.
\end{proof}

Since we often want to localize the model category of marked object, it is useful to have a simple description of homotopy function complexes.
Let $X$ and $Y$ be a pair of marked objects such that $X$ is cofibrant and $Y$ is fibrant.
Let $QU(X)^\bullet$ be a cosimplicial frame on $U(X)$.
Then vertices of the Kan complex $\Map(U(X),U(Y)) = \Hom(QU(X)^\bullet,U(Y))$ can be identified with maps $U(X) \to U(Y)$.
We will denote this bijection by $e : \Hom(\Delta^0,\Map(U(X),U(Y))) \simeq \Hom(U(X),U(Y))$.

Let $x$ and $y$ be a pair of vertices of $\Map(U(X),U(Y))$ that belongs to the same component.
If $e(x) : U(X) \to U(Y)$ lifts to a map of marked objects, then so does the map $e(y)$.
Indeed, there exists an edge $\Delta^1 \to \Map(U(X),U(Y))$ between $x$ and $y$.
But such an edge corresponds to a homotopy between $e(x)$ and $e(y)$.
Since $Y$ is fibrant and $e(x)$ preserves marked maps, so does every map homotopic to $e(x)$.

Let $\Map(X,Y)$ be the Kan subcomplex of $\Map(U(X),U(Y))$ which consists of those components of
$\Map(U(X),U(Y))$ in which some (and hence every) vertex corresponds to a map $U(X) \to U(Y)$ that lifts to a map of marked objects.

\begin{lem}[mark-map]
Kan complex $\Map(X,Y)$ is equivalent to the homotopy function complex from $X$ to $Y$.
\end{lem}
\begin{proof}
We may assume that $QU(X)^\bullet$ is a fibrant cosimplicial frame on $U(X)$.
Then we can define a fibrant cosimplicial frame $QX^\bullet$ on $X$ as $(QU(X)^\bullet,\mathcal{E})$,
where $\mathcal{E}$ consists of those maps $K \to QU(X)^\bullet$ such that $K \to QU(X)^\bullet \to U(X)$ is marked.
Map $L_n QX^\bullet \to QX^n$ is a cofibration since its underlying map $L_n QU(X)^\bullet \to QU(X)^n$ is a cofibration.
Map $QX^n \to X$ has RLP with respect to $\I^\flat$ since $QU(X)^\bullet$ is fibrant.
It also has RLP with respect to $K^\flat \to (K, G \{id\})$ by definition of $QX^n$.
Thus it is a trivial fibration.

Kan complex $\Hom(QX^\bullet,Y)$ is a subcomplex of $\Map(X,Y)$.
Let us prove that they are actually equal.
Consider a simplex $\Delta^n \to \Map(X,Y)$.
Such simplex corresponds to a map $QU(X)^n \to U(Y)$.
We need to show that this map lifts to a map $QX^n \to Y$.
Let $k : K \to QU(X)^n$ be a marked map.
Then it is homotopic to $K \xrightarrow{k} QU(X)^n \to U(X) \simeq QU(X)^0 \to QU(X)^n$,
where the last map is given by any map $\Delta^0 \to \Delta^n$.
By definition of $\Map(X,Y)$ every such map gives a marked map in $Y$.
Since $Y$ is fibrant, $K \xrightarrow{k} QU(X)^n \to U(Y)$ is also marked.
\end{proof}

\section{Marked simplicial sets}
\label{sec:marked-simp-sets}

In this section for every set $\mathcal{L}$ of simplicial sets, we define a model category $\sSet_\mathcal{L}$ of marked simplicial sets
which represents the $(\infty,1)$-category of $(\infty,1)$-categories which have a limit for every diagram of every shape from $\mathcal{L}$.
In particular, if $\mathcal{L} = \{ \varnothing, \Lambda^2_2 \}$, then $\sSet_\mathcal{L}$ represents the $(\infty,1)$-category of finitely complete $(\infty,1)$-categories.

For every set of simplicial sets $\mathcal{L}$, we define set $\mathcal{K}$ as $\{ \Delta^0 \join L\ |\ L \in \mathcal{L} \}$.
Let $\sSet_\mathcal{L}$ be the category of \K-marked objects.
Let $\J$ be the following set of maps of $\sSet_\mathcal{L}$:
\begin{enumerate}
\item For every $L \in \mathcal{L}$, maps $L^\flat \to (\Delta^0 \join L, \{ id : \Delta^0 \join L \to \Delta^0 \join L \})$.
\item For every $n > 0$ and $L \in \mathcal{L}$, maps $(\partial \Delta^n \join L, \{ \Delta^{\{n\}} \join L : \Delta^0 \join L \to \partial \Delta^n \join L \}) \to (\Delta^n \join L, \{ \Delta^{\{n\}} \join L : \Delta^0 \join L \to \Delta^n \join L \})$.
\end{enumerate}
We define a model structure on $\sSet_\mathcal{L}$ as the left Bousfield localization of the model structure
defined in section~\ref{sec:model-structure} (with $\mathcal{J} = \varnothing$) with respect to $\J$.
The following propositions give us a characterization of fibrant objects and fibrations between them in $\sSet_\mathcal{L}$.

\begin{prop}[mark-fib-obj]
For every marked simplicial set $Z$, the following conditions are equivalent:
\begin{enumerate}
\item \label{it:fib} $Z$ is a fibrant object of $\sSet_\mathcal{L}$.
\item \label{it:exp} $Z$ is a quasi-category which has a limit for every diagram of every shape from $\mathcal{L}$, and a cone is marked in $Z$ if and only if it is a limit cone.
\end{enumerate}
\end{prop}
\begin{proof}
It is easy to see that \eqref{it:fib} implies \eqref{it:exp}.
To prove the converse, we need to define homotopy function complexes for $\sSet_\mathcal{L}$.
First, recall that the inclusion of Kan complexes into quasi-categories has a right adjoint which we denote by $E$ (see, for example, \cite[Proposition~1.2.5.3]{lurie-topos}).
Let $\Fun(X,Y)$ be the internal Hom in the category of simplicial sets.
By \cite[Corollary~3.1.4.4]{lurie-topos}, $E(\Fun(X,Y))$ is a homotopy function complex in the Joyal model structure.
If $X$ and $Y$ are marked simplicial sets, then by \rlem{mark-map}, we can define a homotopy function complex $\Map(X,Y)$ as certain subcomplex of $E(\Fun(U(X),U(Y)))$.

Let $X$ be a marked simplicial set of the form $(A \join L, \{ p \join L \})$ for some $L \in \mathcal{L}$ and $p : \Delta^0 \to A$.
Then we can give an equivalent definition of a homotopy function complex $\Map'(X,Y)$ in terms of fat join.
Let $\Map'(X,Y)$ be the subcomplex of $E(\Fun(A \fjoin L, Y))$ which consists of those components in which some (and hence every) vertex
corresponds to a map $f : A \fjoin L \to Y$ such that $\Delta^0 \fjoin L \xrightarrow{p \fjoin L} A \fjoin L \xrightarrow{f} Y$ is a limit cone.
Since $A \fjoin L \to A \join L$ is a weak equivalence in the Joyal model structure, induced map $\Map(X,Y) \to \Map'(X,Y)$ is a weak equivalence of Kan complexes.

An object $Z$ of $\sSet_\mathcal{L}$ is fibrant if and only if it is a quasi-category, marked maps in $Z$ are stable under homotopy,
and for every map $X \to Y$ in $\J$, induced map of Kan complexes $\Map(Y,Z) \to \Map(X,Z)$ is a weak equivalence.
Let $Z$ be a marked simplicial set that satisfies \eqref{it:exp}.
Since limits cones are stable under homotopy, we just need to prove that $\Map(Y,Z) \to \Map(X,Z)$ is a weak equivalence for every $X \to Y$ in $\J$.
Thus we only need to prove that for every $p : X \to Z$, the fiber of $\Map(Y,Z) \to \Map(X,Z)$ over $p$ is contractible.

First, let us consider maps of the form $L^\flat \to (\Delta^0 \join L, \{ id \})$.
Since $L^\flat$ is flat, $\Map(L^\flat,Z) = E(\Fun(L,U(Z)))$.
For any simplicial sets $K$, $A$ and $Z'$ and every map $p : K \to Z'$, we have the following Cartesian square:
\[ \xymatrix{ \Fun(A,Z'^{/p}) \ar[r] \ar[d] \pb & \Fun(A \fjoin K, Z') \ar[d] \\
              \Delta^0 \ar[r]_-p & \Fun(K,Z')
            } \]
In particular, if we let $K = L$, $A = \Delta^0$ and $Z' = U(Z)$, then we obtain the following Cartesian square:
\[ \xymatrix{ U(Z)^{/p} \ar[r] \ar[d] \pb & \Fun(\Delta^0 \fjoin L, U(Z)) \ar[d] \\
              \Delta^0 \ar[r]_-p & \Fun(L,U(Z))
            } \]
Since $E$ is a right adjoint, it preserves pullbacks.
Hence fiber of $\Map'((\Delta^0 \join L, \{ id \}), Z) \to \Map(L^\flat, Z)$ is the subcomplex of $E(U(Z)^{/p})$
spanned by vertices that correspond to limit cones, and it is easy to see that it is contractible (see, for example, \cite[Lemma~2.11]{szumilo}).

Now, let us consider maps of the form $(\partial \Delta^n \join L, \{ \Delta^{\{n\}} \join L \}) \to (\Delta^n \join L, \{ \Delta^{\{n\}} \join L \})$.
First, let us consider the special case when $L = \varnothing$.
We prove that fibers of $\Map((\Delta^n, \{ \Delta^{\{n\}} \}), Z) \to \Map((\partial \Delta^n, \{ \Delta^{\{n\}} \}), Z)$ are contractible by induction on $n$.
If $n = 1$, then for every $p : \Delta^0 \to X$, we have the following diagram:
\[ \xymatrix{ U(Z)^{/p} \ar[d] \ar[r] \pb & \Fun(\Delta^1, U(Z)) \ar[d] \\
              U(Z) \ar[d] \ar[r] \pb & \Fun(\partial \Delta^1, U(Z)) \ar[d] \\
              \Delta^0 \ar[r]_p & U(Z)
            } \]
If $p$ is final, then $U(Z)^{/p} \to U(Z)$ is a weak equivalence (see, for example, \cite[Corollary~1.2.12.5]{lurie-topos}).
Since every map $\Delta^0 \to \Map((\partial \Delta^1, \{ \Delta^{\{1\}} \}), Z)$ factors through $U(Z)$ for some final $p$,
fibers of $\Map((\Delta^1, \{ \Delta^{\{1\}} \}), Z) \to \Map((\partial \Delta^1, \{ \Delta^{\{1\}} \}), Z)$ are contractible.

Now, assume $n > 1$.
Consider the following diagram:
\[ \xymatrix{ \Map((\Delta^n, \{ \Delta^{\{n\}} \}), Z) \ar[d] \\
              \Map((\partial \Delta^n, \{ \Delta^{\{n\}} \}), Z) \ar[r] \ar[d] \pb & \Map((\Delta^{n-1}, \{ \Delta^{\{n-1\}} \}), Z) \ar[d] \\
              \Map((\Lambda^n_1, \{ \Delta^{\{n\}} \}), Z) \ar[r] & \Map((\partial \Delta^{n-1}, \{ \Delta^{\{n-1\}} \}), Z)
            } \]
Since $\Lambda^n_1 \to \Delta^n$ is a weak equivalence in the Joyal model structure, map
\[ \Map((\Delta^n, \{ \Delta^{\{n\}} \}), Z) \to \Map((\Lambda^n_1, \{ \Delta^{\{n\}} \}), Z) \]
is also a weak equivalence.
By induction hypothesis, $\Map((\Delta^{n-1}, \{ \Delta^{\{n-1\}} \}), Z) \to \Map((\partial \Delta^{n-1}, \{ \Delta^{\{n-1\}} \}), Z)$ is a weak equivalence.
Since it is a fibration, map
\[ \Map((\partial \Delta^n, \{ \Delta^{\{n\}} \}), Z) \to \Map((\Lambda^n_1, \{ \Delta^{\{n\}} \}), Z) \]
is also a weak equivalence.
Hence by 2-out-of-3 property, $\Map((\Delta^n, \{ \Delta^{\{n\}} \}), Z) \to \Map((\partial \Delta^n, \{ \Delta^{\{n\}} \}), Z)$ is a weak equivalence.

Finally, we consider the general case.
For every $p : L \to U(Z)$, we have the following diagram:
\[ \xymatrix{ \Fun(\Delta^n, U(Z)^{/p}) \ar[r] \ar[d] \pb & \Fun(\Delta^n \fjoin L, U(Z)) \ar[d] \\
              \Fun(\partial \Delta^n, U(Z)^{/p}) \ar[r] \ar[d] \pb & \Fun(\partial \Delta^n \fjoin L, U(Z)) \ar[d] \\
              \Delta^0 \ar[r]_-p & \Fun(L,U(Z))
            } \]
It follows that the following square is Cartesian:
\[ \xymatrix{ \Map(\Delta^n, U(Z)^{/p}) \ar[r] \ar[d] \pb & \Map'(\Delta^n \join L, U(Z)) \ar[d] \\
              \Map(\partial \Delta^n, U(Z)^{/p}) \ar[r] & \Map'(\partial \Delta^n \join L, U(Z))
            } \]
Since every vertex $\Delta^0 \to \Map'(\partial \Delta^n \join L, U(Z))$ factors through $\Map(\partial \Delta^n, U(Z)^{/p})$
for some $p$, it is enough to show that every fiber of \[ \Map(\Delta^n, U(Z)^{/p}) \to \Map(\partial \Delta^n, U(Z)^{/p}) \] is contractible.
But this follows from the special case when $L = \varnothing$.
\end{proof}

\begin{prop}[mark-fib-map]
A map $f : X \to Y$ between fibrant marked simplicial sets is a fibration if and only if its underlying map is a fibration in the Joyal model structure.
\end{prop}
\begin{proof}
\Rprop{mark-fib-obj} implies that an object of $\sSet_\mathcal{L}$ is fibrant if and only if it is a quasi-category
and has RLP with respect to sets $\J$ and $\J_\mathcal{K}$ (which was defined in \rthm{mark-main}).
By \cite[Proposition~3.6]{f-model-structures}, a map of marked simplicial sets with fibrant codomain is a fibration
if and only if it is a fibration in the Joyal model structure and has RLP with respect to $\J \cup \J_\mathcal{K}$.
Let us prove that every map $f : X \to Y$ between fibrant marked simplicial sets which is a fibration in the Joyal model structure has RLP with respect to $\J \cup \J_\mathcal{K}$.

Since $X$ has RLP with respect to $\J_\mathcal{K}$, $f$ also has RLP with respect to this set.
Consider the following commutative square:
\[ \xymatrix{ L \ar[r] \ar[d] & U(X) \ar[d]^f \\
              \Delta^0 \join L \ar[r]_v \ar@{-->}[ur]^g & U(Y),
            } \]
where $v$ is marked in $Y$.
Since $X$ is fibrant, there exists a marked map $g : \Delta^0 \join L \to U(X)$ such that the top triangle in the square above commutes.
Since $f \circ g$ and $v$ are limit cones, they are homotopic relative to $L$.
Since $U(f)$ is a fibration in the Joyal model structure, we can lift this homotopy to obtain a lift in the square above such that both triangles commute.
The same argument shows that $f$ has RLP with respect to maps $(\partial \Delta^n \join L, \{ \Delta^{\{n\}} \join L \}) \to (\Delta^n \join L, \{ \Delta^{\{n\}} \join L \})$.
\end{proof}

Let $\kappa$ be a regular cardinal, and let $\mathcal{L}$ be the set of all (representatives of) sets with cardinality less than $\kappa$ together with simplicial set $\Lambda^2_2$.
Let $\sSet_\mathcal{L}^f$ be the full subcategory of $\sSet_\mathcal{L}^f$ on the fibrant object.
Then $\sSet_\mathcal{L}^f$ is a category with fibrations.
Let $\qcat^c$ be the category of $\kappa$-complete quasi-categories and functors that preserve $\kappa$-small limits.
A structure of a category with fibrations on $\qcat^c$ was constructed in \cite{szumilo}.
These two categories with fibrations are isomorphic.
Indeed, \rprop{mark-fib-obj} implies that forgetful functor $U : \sSet_\mathcal{L}^f \to \qcat^c$ is an isomorphism of categories.
Thus we just need to show that $U$ preserves weak equivalences and fibrations.
Weak equivalences in both cases are just weak equivalences of underlying quasi-categories.
Fibrations in $\qcat^c$ are fibrations of underlying quasi-categories.
By \rprop{mark-fib-map}, this is also true in $\sSet_\mathcal{L}^f$.

\bibliographystyle{amsplain}
\bibliography{ref}

\end{document}